\documentclass[reqno]{amsart}

\newcommand{\RR}{\mathbb{R}} %Real R
\newcommand{\CC}{\mathbb{C}} %Complex C

\newcommand{\RE}{\mathrm{Re}\,}

\newcommand{\Aut}{\mathrm{Aut}}

\newcommand{\inp}{\,\raisebox{0.4ex}{$\lrcorner$}\;}

\newcommand{\UB}{\mathbb{B}}

\newcommand{\KE}{K\"ahler-Einstein }

\newcommand{\Ric}{\mathrm{Ric}}

\newcommand{\oKE}{\omega_{\mathrm{KE}}}
\newcommand{\oB}{\omega_{\mathrm{B}}}

\newcommand{\dc}{d^c}

\newcommand{\ve}{\varepsilon}

\newcommand{\dz}{dz}

\newcommand{\paren}[1]{\left(#1\right)}
\newcommand{\abs}[1]{\left\lvert#1\right\rvert}
\newcommand{\norm}[1]{\left\|#1\right\|}
\newcommand{\set}[1]{\left\{#1\right\}}

\newcommand{\pd}[2]{\frac{\partial#1}{\partial#2}}

\newcommand{\pdl}[2]{\partial#1/\partial#2}

\newcommand{\ind}[4]{{#1}^{\phantom{#2}#3}_{#2\phantom{#3}#4}}

\newtheorem{theorem}{Theorem}[section]

\newtheorem{lemma}[theorem]{Lemma}
\newtheorem{proposition}[theorem]{Proposition}

\theoremstyle{definition}

\theoremstyle{remark}
\newtheorem{remark}[theorem]{Remark}
\newtheorem{example}[theorem]{Example}

\newcommand{\vW}{\mathcal{W}}
\newcommand{\vZ}{\mathcal{Z}}
\newcommand{\vV}{\mathcal{V}}

\numberwithin{equation}{section}

\title[A method of potential scaling]{A method of potential scaling in the study of  pseudoconvex domains with noncompact automorphism group}

\author{Kang-Hyurk Lee}

\address{Department of Mathematics and Research Institute of Natural Science, Gyeongsang National University, Jinju, Gyeongnam, 52828, Korea}

\email{nyawoo@gnu.ac.kr}


\thanks{The research was supported by the National Research Foundation of Korea (NRF) grant funded by the Korea government (No. NRF-2019R1F1A1060891)}


\begin{document}

\begin{abstract}
The affine scaling method has been a typical approach to  study complex domains with noncompact automorphism group. In this article, we will introduce an alternative approach, so called, the method of potential scaling to construct a certain class of potential functions of the \KE metric. We will also prove that if a bounded pseudoconvex domain admits a  potential function of the \KE metric whose differential has constant length, then there is an $1$-parameter family of automorphisms.
\end{abstract}

\maketitle

\section{Introduction}

For a bounded domain $\Omega$ in the complex Euclidean space $\CC^n$, the \emph{automorphism group} of $\Omega$, denoted by $\Aut(\Omega)$, is the set of automorphisms (self-biholomorphisms) of $\Omega$ under the law of the mapping composition. The automorphism group $\Aut(\Omega)$ with the compact-open topology has a Lie group structure. A fundamental problem in both Complex Analysis and Complex Geometry is the classification of pseudoconvex domains with noncompact automorphism group; especially domains with compact quotient. A fundamental result of the classification is B.~Wong's theorem in \cite{Wong}: a strongly pseudoconvex bounded domain with noncompact automorphism group is biholomorphic to the unit ball. Due to J.P.~Rosay's improvement~\cite{Rosay}, the unit ball is the biholomorphically unique, smoothly bounded domain with compact quotient (see also \cite{GKK}). After S.~Pinchuk's observation (\cite{Pinchuk1980,Pinchuk1989}), the affine scaling method has been a typical approach to classify such domains (see \cite{Bedford-Pinchuk1988, KTK1992, KTK-Krantz2008}). The main application of the affine scaling method is to show the existence of $1$-parameter family of automorphisms. This is an important ingredient in S.~Frankel's study on convex domains with compact quotient (\cite{Frankel}).

The affine scaling method is to construct a biholomorphism from a domain with noncompact automorphism group to an unbounded domain admitting an affine translation as its holomorphic tranformation. Suppose that $\Aut(\Omega)$ is noncompact, equivalently, there exists a sequence of automorphisms $\set{f_j}$ of $\Omega$ such that an automorphism orbit $\set{f_j(p)}$ for some $p\in\Omega$ accumulates at a boundary point of $\Omega$. Then any subsequential limit of $\set{f_j}$ is a holomorphic mapping from $\Omega$ to the boundary $\partial\Omega$; thus it is not a biholomorphic imbedding of $\Omega$ into $\CC^n$ anymore. In the scaling method,  taking an affine mapping $A_j$ of $\CC^n$ whose jacobian $dA_j$ blows up properly, we can make  $\set{A_j\circ f_j}$ to converge subsequentially to a biholomorphic imbedding. The choice of affine mappings is strongly depends on the boundary geometry at an orbit accumulation point and the boundary behavior of an automorphism orbit.

In this paper, we will introduce  a method of potential scaling to construct a certain class of potential functions for the holomorphically invariant K\"ahler metric; especially the complete K\"ahler-Einstein metric (Section~\ref{sec:potential scaling}).  If the potential function we obtained satisfies a specified condition, we can construct an $1$-parameter family of automorphisms (Theorem~\ref{thm:main thm2}). This is an alternative (but fundamentally same) approach to the affine scaling method in the study of domains with noncompact automorphism group. We will also deal with a relation to the affine scaling method.

\medskip

\noindent\emph{Notation and Convention.} Throughod paper, the summation convention for duplicated indices is always assumed. We denote the complex conjugate of a tensor by taking the bar on the indices, that is, $\overline{z^\alpha} = z^{\bar\alpha}$, $\overline{h_{\alpha\bar\beta}} = h_{\bar\alpha\beta}$ and so on.

%%%%%%%%%%%%%%%%%%%%%%%%%%%%%%%%%%%%%%%%%%%%%%%%%
%%%%%%%%%%%%%%%%%%%%%%%%%%%%%%%%%%%%%%%%%%%%%%%%%
%%%%%%%%%%%%%%%%%%%%%%%%%%%%%%%%%%%%%%%%%%%%%%%%%
%%%%%%%%%%%%%%%%%%%%%%%%%%%%%%%%%%%%%%%%%%%%%%%%%
%%%%%%%%%%%%%%%%%%%%%%%%%%%%%%%%%%%%%%%%%%%%%%%%%
%%%%%%%%%%%%%%%%%%%%%%%%%%%%%%%%%%%%%%%%%%%%%%%%%

\section{The method of potential scaling and main results}\label{sec:potential scaling}

In this section, we introduce the method of potential scaling and main results. The potential scaling is to rescale potential functions of the holomorphically invariant K\"ahler metric of a bounded pseudoconvex domain by automorphisms. Typical invariant K\"ahler metrics are the complete K\"ahler-Einstein metric and the Bergman metric. In this paper, we will focus on the \KE metric. At the last of this section, we will also discuss the potential scaling for the Bergman metric.

%%%%%%%%%%%%%%%%%%%%%%%%%%%%%%%%%%%%%%%%
%%%%%%%%%%%%%%%%%%%%%%%%%%%%%%%%%%%%%%%%
%%%%%%%%%%%%%%%%%%%%%%%%%%%%%%%%%%%%%%%%

\subsection{The \KE metric} 
For a bounded pseudoconvex domain $\Omega$ in $\CC^n$, the \emph{\KE metric} of $\Omega$, denoted by its K\"ahler form  $\oKE$, is the unique complete K\"ahler metric with the normalized Einstein condition,
\begin{equation*}
\Ric_{\oKE}=-(n+1)\oKE \;.
\end{equation*}
The uniqueness is due to Yau's Schwarz lemma in \cite{Yau1978} and the existence is due to Cheng-Yau~\cite{Cheng-Yau} and Mok-Yau~\cite{Mok-Yau}. Since $\Ric_{\oKE}=-dd^c \log\det(h_{\alpha\bar\beta})$ where $\oKE=ih_{\alpha\bar\beta}dz^\alpha\wedge dz^{\bar\beta}$ in the standard coordinates $z=(z^1,\ldots,z^n)$ and $\dc=\displaystyle{\frac{i}{2}(\bar\partial-\partial)}$,  we can write the Einstein condition by $dd^c \log\det(h_{\alpha\bar\beta})=(n+1)\oKE$. For the sake of simplicity, we will denote by
\begin{equation*}
\psi = \det(h_{\alpha\bar\beta}) 
\end{equation*}
throughout this paper. Then the Einstein condition is now of the form
\begin{equation*}%\label{eqn:KE potential}
dd^c \log\psi=(n+1)\oKE \;.
\end{equation*}
Thus the function $\log\psi$ is a canonical potential function of $\oKE$. By Yau's Schwarz lemma, each automorphism $f\in\Aut(\Omega)$ preserves the volume form so that 
\begin{equation*}
f^*(\oKE)^n=(\oKE)^n
\end{equation*}
equivalently
\begin{equation}\label{eqn:KE transformation}
(\psi\circ f) \abs{J_f}^2=\psi \;.
\end{equation}
This implies that each $f\in\Aut(\Omega)$ satisfies  
\begin{multline*}
(n+1)f^*\oKE 
=f^* dd^c \log\psi 
=dd^c  \log(\psi\circ f) \\
=dd^c \paren{\log\psi -\log\abs{J_f}^2} 
= dd^c \log\psi
= (n+1)\oKE\;, 
\end{multline*}
so is an isometry of $\oKE$. Here $J_f$ is the holomorphic Jacobian determinant
\begin{equation*}
J_f=\det\paren{\pd{f^\alpha}{z^\beta}}
\end{equation*}
of the holomorphic mapping $f=(f^1,\ldots,f^n)$.

%%%%%%%%%%%%%%%%%%%%%%%%%%%%%%%%%%%%%%%%
%%%%%%%%%%%%%%%%%%%%%%%%%%%%%%%%%%%%%%%%
%%%%%%%%%%%%%%%%%%%%%%%%%%%%%%%%%%%%%%%%

\subsection{The scaling of \KE potentials}
Let $f\in\Aut(\Omega)$ and consider the pulling-back function 
\begin{equation*}
f^*\log\psi = \log(\psi\circ f)
\end{equation*} 
which is also a potential function of $\oKE$ for each $f\in\Aut(\Omega)$ by \eqref{eqn:KE transformation}

Suppose that there is a sequence $\{f_j\}$ of automorphisms whose orbit of a point of $\Omega$ accumulates at a boundary point. Since any subsequential limit $\set{f_j(p)}$ for each $p\in\Omega$ is on the boundary $\partial\Omega$, the completeness of  $\oKE$ implies that the sequence of potential functions, $\set{\log(\psi\circ f_j)}$, blows up in the sense that $\lim_{j\to\infty}\log(\psi\circ f_j)=\infty$. \emph{The method of potential scaling} is to take dominators $c_j$ properly so that potential functions
\begin{equation*}%\label{eqn:potential scaling KE}
\log\frac{\psi\circ f_j}{c_j}
\end{equation*}
converges to another potential function of $\oKE$. The following theorem is on the convergence of the potential scaling.

\begin{theorem}\label{thm:main thm1}
Let $\oKE=ih_{\alpha\bar\beta}dz^\alpha\wedge dz^{\bar\beta}$ be the \KE metric of the bounded pseudoconvex domain $\Omega$ and let $\psi= \det(h_{\alpha\bar\beta})$. Suppose that there is a constant $C>0$ with
\begin{equation}\label{cond:KE KH}
\norm{d\log\psi}_{\oKE}^2 
	= \pd{\log\psi}{z^\alpha}\pd{\log\psi}{z^{\bar\beta}} h^{\alpha\bar\beta}
	<C^2 
	\quad\text{on $\Omega$.}
\end{equation}
Then for any compact subset $K$ in $\Omega$, the collection of holomorphic functions,
\begin{equation*}
\mathcal{F}_K
	=\set{ \frac{J_f}{J_f(p)}: f\in\Aut(\Omega), p\in K} \;,
\end{equation*}
is a normal family. Moreover any limit of a convergent sequence in $\mathcal{F}_K$ is a nowhere vanishing holomorphic function.
\end{theorem}
Here $(h^{\alpha\bar\beta})$ stands for the inverse matrix of the \KE metric $(h_{\alpha\bar\beta})$.  Assumption~\eqref{cond:KE KH} is associated with the K\"ahler-hyperbolicity of $\oKE$ as will be mentioned in Remark~\ref{rmk:KH}. We will prove this theorem in Section~\ref{sec:convergence}.

\medskip

For a compact subset $K$ in $\Omega$, take $f_j\in\Aut(\Omega)$ and $p_j\in K$ for each $j$. Under the assumption of Theorem~\ref{thm:main thm1}, let us consider a potential scaling
\begin{equation*}%\label{eqn:potential scaling KE explicit}
\log \frac{\psi\circ f_j}{(\psi\circ f_j)(p_j)}\;.
\end{equation*}
Since
\begin{equation*}
\frac{\psi\circ f_j}{(\psi\circ f_j)(p_j)} 
	= \frac{\psi}{\psi(p_j)}  \abs{\frac{J_{f_j}(p_j)}{J_{f_j}}}^2
\end{equation*}
by \eqref{eqn:KE transformation}, Theorem~\ref{thm:main thm1} implies that the sequence $\set{J_{f_j}/J_{f_j}(p_j)}$ admits a subsequence converging uniformly on any compact subset of $\Omega$. Passing to a subsequence, we may assume that $J_{f_j}/J_{f_j}(p_j)\to\eta$ uniformly on any compact subset of $\Omega$  and $p_j\to p\in\Omega$. Then
\begin{equation}\label{eqn:scaling limit}
\frac{\psi\circ f_j}{(\psi\circ f_j)(p_j)} 
	\to  \psi_\infty=\frac{\psi}{\psi(p)}\frac{1}{\abs{\eta}^2}
\end{equation}
in the local $C^\infty$ topology. Therefore $\log\psi_\infty$ satisfies $dd^c\log\psi_\infty=(n+1)\oKE$, so it is also a potential function of $\oKE$. 

\medskip

The function $\log\psi_\infty$ possesses information on the boundary value of $\norm{d\log\psi}^2_{\oKE}$.

\begin{proposition}\label{prop:boundary behavior}
Assume \eqref{eqn:scaling limit}. Then 
\begin{equation*}
\norm{d\log\psi_\infty}_{\oKE}(p) 
	=\lim_{j\to\infty} \norm{\log \psi}_{\oKE}(f_j(p)) 
\end{equation*}
for any $p\in\Omega$.
\end{proposition}

\begin{proof}
Since each $f_j$ is an isometry of $\oKE$, we have
\begin{multline*}
\norm{d\log\psi_\infty}_{\oKE}(p) 
	= \lim_{j\to\infty} \norm{\log\frac{\psi\circ f_j}{(\psi\circ f_j)(p_j)}}_{\oKE}(p) \\
	= \lim_{j\to\infty} \norm{\log \psi\circ f_j}_{\oKE}(p)
	= \lim_{j\to\infty} \norm{\log \psi}_{\oKE}(f_j(p)) 
\end{multline*}
for any $p\in\Omega$. 
\end{proof}

Let us see the boundary behavior of $\norm{\log \psi}^2_{\oKE}$ of the unit ball.

\begin{example}
For the unit ball $\UB^n=\set{z\in\CC^n: \norm{z}<1}$, the \KE metric  $\oKE^{\UB^n}=ih^{\UB^n}_{\alpha\bar\beta}dz^\alpha\wedge dz^{\bar\beta}$ is given by
\begin{equation*}
h^{\UB^n}_{\alpha\bar\beta}
	=\frac{1}{\paren{1-\norm{z}^2}^2}
	\paren{
		\delta_{\alpha\bar\beta}(1-\norm{z}^2)
		+z^{\bar\alpha}z^\beta
	} \;,
\end{equation*}
and its inverse is given by
\begin{equation*}
(h^{\UB^n})^{\alpha\bar\beta}
	=(1-\norm{z}^2)
	\paren{
		\delta^{\alpha\bar\beta}-z^\alpha z^{\bar\beta}
	}
	\;.
\end{equation*}
For the determinant of the metric tensor
\begin{equation*}
\psi^{\UB^n}
	=\det(h^{\UB^n}_{\alpha\bar\beta})
	=\frac{1}{\paren{1-\norm{z}^2}^{n+1}}
\;,
\end{equation*}
we can easily see $d\dc\log\psi^{\UB^n}=(n+1)\oKE^{\UB^n}$. Then we have
\begin{equation*}
\norm{d\log\psi^{\UB^n}}_{\oKE}^2
	=(n+1)^2\norm{z}^2
	\;.
\end{equation*}
This implies that the boundary value of $\norm{d\log\psi^{\UB^n}}_{\oKE}$ is $n+1$.
\end{example}

In the case of a strongly pseudconvex domain $\Omega$ with $C^\infty$ smooth boundary, the boundary behavior of the geometric quantities of $\log\psi$ is the same as that of the unit ball (see \cite{Cheng-Yau,Coevering,Choi}). Thus function $\norm{d\log \psi}_{\oKE}$ is continuous up to the boundary of $\Omega$ and its boundary value is always $n+1$:
\begin{equation*}
\lim_{p\to\partial\Omega} \norm{\log \psi}_{\oKE}(p) 
	= n+1 \;.
\end{equation*}
Proposition~\ref{prop:boundary behavior} implies that if an orbit of $\{f_j\}$ accumulate at a boundary point of $\Omega$, then the potential scaling limit $\log\psi_\infty$ satisfies
\begin{equation*}
\norm{d\log\psi_\infty}_{\oKE}\equiv n+1 \;.
\end{equation*}

The second main result of this paper is on the existence of $1$-parameter family of automorphisms.
\begin{theorem}\label{thm:main thm2}
Let $\Omega$ be a bounded pseudoconvex domain in $\CC^n$. If there is a positive-valued smooth function $\tilde\psi:\Omega\to\RR$ such that
\begin{equation*}
dd^c\log\tilde\psi=(n+1)\oKE \quad\text{and}\quad \|d\log\tilde\psi\|_{\oKE}\equiv C
\end{equation*}
for some positive constant $C\leq n+1$, then there is a nowhere vanishing complete holomorphic vector field on $\Omega$.
\end{theorem}

By a holomorphic tangent vector field, we means a holomorphic section $\vZ$ to the holomorphic tangent bundle $T^{1,0}\Omega$. If the corresponding real tangent vector field $\RE\vZ=\vZ+\overline{\vZ}$ is complete, we also say $\vZ$ is complete. Thus a complete holomorphic tangent vector field in Theorem~\ref{thm:main thm2} generates an $1$-parameter family of holomorphic automorphisms of $\Omega$. The proof will be in Section~\ref{sec:existence}.

\medskip

From Kai-Ohsawa~\cite{Kai-Ohsawa}, every bounded homogeneous domain  (equivalently, a bounded pseudoconvex domain biholomorphic to an affine homogeneous domain) has a potential function $\log\tilde\psi$ of $\oKE$ such that $\|d\log\tilde\psi\|_{\oKE}$ is constant and the constant is uniquely determined by the complex structure of the domain. But most bounded homogeneous domains in $\CC^n$ except the unit ball should have the constant greater than $n+1$. In the $1$-dimensional case, the possible constant is only $2=1+1$.
This case have been dealt in \cite{CLY}:
\begin{theorem}[Theorem 3.1 in \cite{CLY}]
\textit{Let $X$ be a Riemann surface with  a complete hermitian metric $\omega_X$ with constant Gaussian curvature $-4$. If there is a function $\varphi:X\to\RR$ with 
\begin{equation*}
d\dc\log\varphi = 2\omega_X 
\quad\text{and}\quad
\norm{d\log\varphi}_{\omega_X}\equiv 2
\end{equation*}
then there is a nowhere vanishing complete holomorphic vector field on $X$.
}
\end{theorem}

%%%%%%%%%%%%%%%%%%%%%%%%%%%%%%%%%%%%%%%%
%%%%%%%%%%%%%%%%%%%%%%%%%%%%%%%%%%%%%%%%
%%%%%%%%%%%%%%%%%%%%%%%%%%%%%%%%%%%%%%%%

\subsection{Affine scaling limits and potential scaling limits} 
As we mentioned in Introduction,  the affine scaling method for a sequence $\set{f_j}$ of automorphisms of $\Omega$ is to take affine mapping $A_j$ so that $\set{A_j\circ f_j}$ converges to a holomorphic imbedding. The scaling method due to S.~Frankel~\cite{Frankel} is to choose $A_j$ as
\begin{equation*}
A_j(z)= (df_j(p_j))^{-1}(z-f_j(p_j))
\end{equation*}
where $p_j$ lies on a fixed compact subset $K$ of $\Omega$. If $\Omega$ is convex or the boundary of an orbit accumulating point is locally convex, then the scaling sequence
\begin{equation} \label{eqn:Frankel scaling}
A_j\circ f_j (z) 
	= (df_j(p_j))^{-1}(f_j(z)-f_j(p_j))
\end{equation}
converges to a biholomorphic imbedding (see  K.T.~Kim~\cite{KTK1990}).

\medskip

Suppose that $\set{A_j\circ f_j}$ in \eqref{eqn:Frankel scaling} converges to a biholomorphism $F:\Omega\to\Omega'$. Since each holomorphic Jacobian determinant of $A_j\circ f_j$  is of the form
\begin{equation*}
J_{A_j\circ f_j}
	=J_{A_j}J_{f_j}=\frac{J_{f_j}}{J_{f_j}(p_j)} \;,
\end{equation*}
we have
\begin{equation*}
\frac{J_{f_j}}{J_{f_j}(p_j)}
	\to J_F \;.
\end{equation*}
This implies that $J_F$ is the same as $\eta$ in \eqref{eqn:scaling limit}. 

Let $\oKE=ih_{\alpha\bar\beta}dz^\alpha\wedge dz^{\bar\beta}$ and $\oKE'=ih'_{\alpha\bar\beta}dz^\alpha\wedge dz^{\bar\beta}$ be the \KE metrics of $\Omega$ and $\Omega'$, respectively. Since $F^*\oKE'=\oKE$, we have
\begin{equation*}
(\psi'\circ F) \abs{J_F}^2=\psi
\end{equation*}
where $\psi=\det(h_{\alpha\bar\beta})$ and $\psi'=\det(h'_{\alpha\bar\beta})$. The scaling limit $\psi_\infty$ as in  \eqref{eqn:scaling limit} is indeed the pulling-back of $\psi'$ by $F$ in the sense of
\begin{equation*}
\psi_\infty 
	= \frac{\psi}{\psi(p)}\frac{1}{\abs{\eta}^2} 
	= \frac{\psi}{\psi(p)}\frac{1}{\abs{J_F}^2} 
	= \frac{1}{\psi(p)}\psi'\circ F 
	\;.
\end{equation*}

As a conclusion, the potential scaling limit $\log\psi_\infty$ is the pulling-back of the canonical potential function $\psi'=\det(h'_{\alpha\bar\beta})$ of the limit domain $\Omega'$ by the affine scaling limit $F$.

%%%%%%%%%%%%%%%%%%%%%%%%%%%%%%%%%%%%%%%%
%%%%%%%%%%%%%%%%%%%%%%%%%%%%%%%%%%%%%%%%
%%%%%%%%%%%%%%%%%%%%%%%%%%%%%%%%%%%%%%%%

\subsection{The scaling of Bergman kernel functions}
For a bounded domain $\Omega$ in $\CC^n$, the \emph{Bergman metric} $\oB$ of $\Omega$ is defined by
\begin{equation*}
\oB=d\dc\log K_\Omega
\end{equation*}
where $K_\Omega=K_\Omega(z,z)$ is the Bergman kernel function of $\Omega$ (\cite{Bergman1950}). By the transformation formular of $K_\Omega$ under $f\in\Aut(\Omega)$, namely
\begin{equation}\label{eqn:Bergman transformation}
(K_\Omega\circ f)\abs{J_f}^2 =K_\Omega \;,
\end{equation}
the Bergman metric is invariant under the action by $\Aut(\Omega)$. Thus every automorphism of $\Omega$ is an isometry with respect to $\oB$. Identity~\eqref{eqn:Bergman transformation} also implies the analogue of Theorem~\ref{thm:main thm1} as following.

\begin{theorem}\label{thm:main thm1'}
Let $\oB=ig_{\alpha\bar\beta}dz^\alpha\wedge dz^{\bar\beta}$ be the Bergman metric of the bounded domain $\Omega$. Suppose that the length of $d\log K_\Omega$ with respect to $\oB$ is bounded, i.e. there is a constant $C>0$ with
\begin{equation}\label{cond:Bergman KH}
\norm{d\log K_\Omega}_{\oB}^2=\pd{\log K_\Omega}{z^\alpha}\pd{\log K_\Omega}{z^{\bar\beta}}g^{\alpha\bar\beta}<C^2
\end{equation}
on $\Omega$. For any compact subset $K$ in $\Omega$, the collection of holomorphic functions,
\begin{equation*}
\mathcal{F}_K=\set{ \frac{J_f}{J_f(p)}: f\in\Aut(\Omega), p\in K} \;,
\end{equation*}
is a normal family. Moreover any limit of a convergent sequence in $\mathcal{F}_K$ is a nowhere vanishing holomorphic function.
\end{theorem}
As same as the scaling of \KE potentials, this theorem implies the convergence of 
\begin{equation*}
\log\frac{K_\Omega\circ f_j}{(K_\Omega\circ f_j)(p_j)} \;,
\end{equation*}
under Assumption~\eqref{cond:Bergman KH}.

\begin{remark}\label{rmk:KH}
Assumptions \eqref{cond:KE KH} and \eqref{cond:Bergman KH} are sufficient conditions for $\oKE$ and $\oB$ to be K\"ahler hyperbolic, respectively. Let $(X,\omega)$ be a $n$-dimensional K\"ahler manifold. If there is a global $1$-form $\theta$ such that $d\theta=\omega$ and the length $\norm{\theta}_\omega$ of $\theta$ with respect to $\omega$ is bounded on $X$, then $(X,\omega)$ is called  \emph{K\"ahler hyperbolic} due to M.~Gromov~\cite{Gromov}. The K\"ahler-hyperbolicity implies the vanishing of the space of $L^2$ harmonic $(p, q)$ forms (see \cite{Donnelly1994} also).

Suppose that there is a global K\"ahler potential $\Phi:X\to\RR$ of $\omega$, i.e. $d\dc\Phi=\omega$. If $\norm{d\Phi}_\omega$ is bounded on $X$, then $\omega$ is K\"ahler hyperbolic because $\norm{\dc\Phi}_\omega=2\norm{d\Phi}_\omega$. For a bounded pseudoconvex domain $\Omega$, the typical invariant K\"ahler metric has a canonical global potential, so the study on the K\"ahler hyperbolicity of domains has been concentrated on the  length of differential of the canonical potentials (see \cite{Donnelly1997,Yeung}).
\end{remark}

%%%%%%%%%%%%%%%%%%%%%%%%%%%%%%%%%%%%%%%%%%%%%%%%%
%%%%%%%%%%%%%%%%%%%%%%%%%%%%%%%%%%%%%%%%%%%%%%%%%
%%%%%%%%%%%%%%%%%%%%%%%%%%%%%%%%%%%%%%%%%%%%%%%%%
%%%%%%%%%%%%%%%%%%%%%%%%%%%%%%%%%%%%%%%%%%%%%%%%%
%%%%%%%%%%%%%%%%%%%%%%%%%%%%%%%%%%%%%%%%%%%%%%%%%
%%%%%%%%%%%%%%%%%%%%%%%%%%%%%%%%%%%%%%%%%%%%%%%%%

\section{Convergence of potential scaling}\label{sec:convergence}

In order to prove Theorem~\ref{thm:main thm1} and Theorem~\ref{thm:main thm1'}, we have the following basic estimate.
\begin{lemma}\label{lemma:boundedness}
Let $\Omega$ be a domain in $\CC^n$ with a holomorphically invariant K\"ahler metric $\omega$.  Suppose that there is a positive-valued function $\varphi:\Omega\to\RR$ satisfying
\begin{equation}\label{eqn:boundedness}
\norm{d\log\varphi}_\omega<C \quad\text{on $\Omega$}
\end{equation}
for some $C>0$. Then for each compact subsets $K$ and $K'$ in $\Omega$, there is $A>0$ such that
\begin{equation*}
\frac{1}{A}\leq\frac{\varphi\circ f}{(\varphi\circ f)(p)}\leq A \quad\text{on $K'$}
\end{equation*}
for any $f\in\Aut(\Omega)$ and $p\in K$.
\end{lemma}

\begin{proof}
Let $f\in\Aut(\Omega)$ and $p\in K$. The automorphism $f$ is isometric with respect to $\omega$ so that $\norm{d(\varphi\circ f)}_\omega^2
		=\norm{d\varphi}_\omega^2\circ f$.
Since $\varphi$ is positive on $\Omega$, we have $d\varphi=\varphi (d\log\varphi)$; hence
\begin{equation}\label{eqn:basic inequality}
\norm{d(\varphi\circ f)}_\omega^2
	= \norm{d\varphi}_\omega^2\circ f
	= (\varphi\circ f)^2 \paren{\norm{d\log\varphi}_\omega^2\circ f}
	\;.
\end{equation}
For the sake of simplicity, let us denote by
\begin{equation*}
\sigma_{f,p}=\frac{\varphi\circ f}{(\varphi\circ f)(p)} \;.
\end{equation*}
Assumption~\eqref{eqn:boundedness} and Identity~\eqref{eqn:basic inequality}  imply that
\begin{align*}
\norm{d\sigma_{f,p}}_\omega^2
	&= \frac{1}{\abs{(\varphi\circ f)(p)}^2}\norm{d(\varphi\circ f)}_\omega^2
	= \abs{\frac{\varphi\circ f}{(\varphi\circ f)(p)}}^2  \paren{\norm{d\log\varphi}_\omega^2\circ f}
	\\
	&\leq  C^2\abs{\sigma_{f,p}}^2 
	\;.
\end{align*}
For a unit speed curve $\gamma:(-R,R)\to\Omega$ with respect to $\omega$ with $\gamma(0)=p$, this inequality implies that the positive-valued function $\sigma_{f,p}\circ\gamma:\RR\to\RR$ satisfies
\begin{equation*}
\abs{(\sigma_{f,p}\circ\gamma)'(t)}\leq C(\sigma_{f,p}\circ\gamma)(t) \;.
\end{equation*}
Since $(\sigma_{f,p}\circ\gamma)(0)=\sigma_{f,p}(p)=1$, Gronwall's inequality gives
\begin{equation*}
e^{-C t}\leq
	(\sigma_{f,p}\circ\gamma)(t)
	\leq e^{Ct} \;.
\end{equation*}
For a point $q\in\Omega$ with $d_\omega(p,q)<R$ where $d_\omega$ is the distance associated to $\omega$, we get
\begin{equation*}
e^{-CR}\leq\sigma_{f,p}(q)\leq e^{CR} \;.
\end{equation*}
This is independent of any choice of $f\in\Aut(\Omega)$ and $p\in K$. Let $K'$ be a compact subset of $\Omega$ and let $R=\sup\{d_\omega(p,q): p\in K, q\in K'\}$. Then we have
\begin{equation*}
e^{-CR}\leq\sigma_{f,p}(q)\leq e^{CR} \quad\text{for any  $q\in K'$}\;.
\end{equation*}
This completes the proof.
\end{proof}

\begin{remark}
In this proof, we only use the fact that $f\in\Aut(\Omega)$ is the isometry of $\omega$. Therefore the estimate in Lemma~\ref{lemma:boundedness} holds for any isometry $f$ of $(\Omega,\omega)$.
\end{remark}

Now we present proofs of Theorem~\ref{thm:main thm1} and Theorem~\ref{thm:main thm1'}.

\medskip

\noindent\textit{Proof of Theorem~\ref{thm:main thm1} and Theorem~\ref{thm:main thm1'}.} We will prove both theorems simultaneously. Let $\omega$ be a invariant K\"ahler metric of $\Omega$:
\begin{enumerate}
\item If $\omega$ is the complete \KE metric $\oKE$, let $\varphi=\det(h_{\alpha\bar\beta})$ where $\oKE=ih_{\alpha\bar\beta}dz^\alpha\wedge dz^{\bar\beta}$.
\item If $\omega$ is the Bergman metric $\oB$, let $\varphi=K_\Omega$.
\end{enumerate}
In both cases, Equations  \eqref{eqn:KE transformation} and  \eqref{eqn:Bergman transformation}  imply that
\begin{equation}\label{eqn:transformation formulae}
(\varphi\circ f) \abs{J_f}^2=\varphi
\end{equation}
for any $f\in\Aut(\Omega)$. 

Now let us assume that
\begin{equation*}
\norm{d\log\varphi}_{\omega}<C \quad\text{on $\Omega$.}
\end{equation*} 
for a constant $C>0$ and let $K$ be a compact subset of $\Omega$. Lemma~\ref{lemma:boundedness} implies that for any compact subset $K'\subset\Omega$, we have $A>0$ such that
\begin{equation*}
\frac{1}{A}
	\leq
	\frac{\varphi\circ f}{\varphi(f(p))}
	\leq A \quad\text{on $K'$,}
\end{equation*}
for any $f\in\Aut(\Omega)$ and $p\in K$. By the transformation formular~\eqref{eqn:transformation formulae}, we have
\begin{equation*}
\frac{1}{A}\leq
	\frac{\varphi\circ f}{\varphi(f(p))}
	=\frac{\varphi}{\varphi(p)} \abs{\frac{J_f(p)}{J_f}}^2
	\leq A
\quad\text{on $K'$.}
\end{equation*}
In other words,
\begin{equation*}
\frac{1}{A}\frac{\varphi}{\varphi(p)} 
\leq\abs{\frac{J_f}{J_f(p)^2}}^2
\leq A\frac{\varphi}{\varphi(p)} \quad\text{on $K'$.}
\end{equation*}
Since the function $\varphi/\varphi(p)$ is pinched by positive constants on $K'$ independent of the choice of $p\in K$, we can conclude that every element of
\begin{equation*}
\mathcal{F}_K=\set{\frac{J_f}{J_f(p)}:f\in\Aut(\Omega),p\in K}
\end{equation*}
is uniformly bounded on $K'$. Since $K'$ is arbitrary, Montel's theorem implies that $\mathcal{F}_K$ is normal.

Since $J_f/J_f(p)=1$ at $p$, a limit holomorphic function of $\set{J_f/J_f(p):f\in\Aut(\Omega)}$ is nowhere vanishing by Hurwitz's theorem. This completes the proof of Theorem~\ref{thm:main thm1} and Theorem~\ref{thm:main thm1'}. \qed

%%%%%%%%%%%%%%%%%%%%%%%%%%%%%%%%%%%%%%%%%%%%%%%%%
%%%%%%%%%%%%%%%%%%%%%%%%%%%%%%%%%%%%%%%%%%%%%%%%%
%%%%%%%%%%%%%%%%%%%%%%%%%%%%%%%%%%%%%%%%%%%%%%%%%
%%%%%%%%%%%%%%%%%%%%%%%%%%%%%%%%%%%%%%%%%%%%%%%%%
%%%%%%%%%%%%%%%%%%%%%%%%%%%%%%%%%%%%%%%%%%%%%%%%%
%%%%%%%%%%%%%%%%%%%%%%%%%%%%%%%%%%%%%%%%%%%%%%%%%

\section{Existence of complete holomorphic vector fields}\label{sec:existence}

In this section, we will prove Theorem~\ref{thm:main thm2}. Throughout this section, we will use the symbol $\Phi$ as a potential function of the \KE metric, instead of $\log\tilde\psi$  as in Theorem~\ref{thm:main thm2}.

\subsection{Covariant derivatives of potential functions}
Let $\Omega$ be a bounded pseudoconvex domain in $\CC^n$ with the \KE metric $\oKE=ih_{\alpha\bar\beta} dz^\alpha \wedge dz^{\bar\beta}$.  Suppose that there is a potential function $\Phi:\Omega\to\RR$ in the sense that $d\dc\Phi= (n+1)\oKE$. We will compute covariant derivatives of 
\begin{equation*}
\norm{d\Phi}_{\oKE}^2=\Phi_\alpha \Phi_{\bar\beta} h^{\alpha\bar\beta} = \Phi_\alpha\Phi^\alpha \;.
\end{equation*}
As above, we will  use the matrix of the K\"aher metric $(h_{\alpha\bar\beta})$ and its inverse matrix $(h^{\bar\beta\alpha})$ to raise and lower indices. We will denote coordinate vector fields by $\partial_\alpha = \pdl{}{z^\alpha}$ and $\partial_{\bar\alpha} = \pdl{}{z^{\bar\alpha}} $ for $\alpha=1,\ldots,n$ and covariant derivatives by $\nabla_A=\nabla_{\partial_A}$ for $A=1,\ldots,n,\bar{1},\ldots,\bar{n}$ where $\nabla$ is the K\"ahler connection of $\oKE$.  Especially covariant derivatives of $\Phi$ will be denoted by $\Phi_{A}=\nabla_A \Phi=\partial_A\Phi$, $\Phi_{AB}=\nabla_B \nabla_A \Phi$ and so on. Since $\Phi$ is real-valued, $\overline{\Phi_{AB\cdots}} = \Phi_{\bar A \bar B \cdots}$.

The \emph{connection form} $(\ind{\theta}{\beta}{\alpha}{})$ of $\oKE$ is the collection of $1$-forms uniquely determined  by the metric compatibility condition
\begin{equation}\label{eqn:metric comp}
dh_{\alpha\bar\beta}
	=\ind{\theta}{\alpha}{\gamma}{}h_{\gamma\bar\beta}+\ind{\theta}{\bar\beta}{\bar\gamma}{}h_{\bar\gamma\alpha}
\quad\text{for any $\alpha,\beta$}
\end{equation}
where $\ind{\theta}{\bar\beta}{\bar\gamma}{}=\overline{\ind{\theta}{\beta}{\gamma}{}}$ and the torsion-free condition
\begin{equation}\label{eqn:torsion free}
0=dz^\beta\wedge\ind{\theta}{\beta}{\alpha}{}
\quad\text{for any $\alpha$.}
\end{equation}
The $1$-form $\ind{\theta}{\beta}{\alpha}{}$ is of type $(1,0)$ so $\ind{\theta}{\beta}{\alpha}{}= (\partial h_{\beta\bar\gamma}) h^{\bar\gamma\alpha}$. Then the K\"ahler connection $\nabla$ can be given by 
\begin{align*}
\nabla \partial_\alpha &= \ind{\theta}{\alpha}{\beta}{}\otimes\partial_\beta\;,
	& \nabla \partial_{\bar\alpha} &= \ind{\theta}{\bar\alpha}{\bar\beta}{}\otimes\partial_{\bar\beta}\;, 
	\\
	\nabla\dz^\alpha&=-\ind{\theta}{\beta}{\alpha}{}\otimes\dz^\beta\;,
	& \nabla\dz^{\bar\alpha}&=-\ind{\theta}{\bar\beta}{\bar\alpha}{}\otimes\dz^{\bar\beta} \;.
\end{align*}
For second order covariant derivatives 
\begin{align*}
\Phi_{\alpha\beta}
	&=\nabla_\beta\nabla_\alpha\Phi
	=\partial_\beta\Phi_\alpha-\ind{\theta}{\alpha}{\gamma}{}(\partial_\beta)\Phi_\gamma
	\;,
	\\
\Phi_{\alpha\bar\beta}
	&=\nabla_{\bar\beta}\nabla_\alpha\Phi
	=\partial_{\bar\beta}\partial_\alpha\Phi=(n+1)h_{\alpha\bar\beta}
\end{align*}
of $\Phi$ can be obtained by
\begin{equation}\label{eqn:second cd}
d\Phi_\alpha-\ind{\theta}{\alpha}{\beta}{}\Phi_\beta
	=\Phi_{\alpha\beta}dz^\beta+\Phi_{\alpha\bar\beta}dz^{\bar\beta}
\;.
\end{equation}
Note that $\Phi_{\alpha\beta}=\Phi_{\beta\alpha}$ since $\nabla$ is torsion-free. Differentiaing Equation~\eqref{eqn:second cd}, we have
\begin{equation*}
-(d\ind{\theta}{\alpha}{\beta}{})\Phi_\beta
+\ind{\theta}{\alpha}{\beta}{}\wedge d\Phi_\beta
=d\Phi_{\alpha\beta}\wedge dz^\beta
+d\Phi_{\alpha\bar\beta}\wedge dz^{\bar\beta}
\;.
\end{equation*}
Applying \eqref{eqn:torsion free} and \eqref{eqn:second cd}, this can be written by 
\begin{multline*}
-\paren{
	d\ind{\theta}{\alpha}{\beta}{}-\ind{\theta}{\alpha}{\gamma}{}\wedge\ind{\theta}{\gamma}{\beta}{}
	}\Phi_\beta
\\
=
\paren{
	d\Phi_{\alpha\lambda}
	-\Phi_{\alpha\beta}\ind{\theta}{\lambda}{\beta}{}
	-\ind{\theta}{\alpha}{\beta}{}\Phi_{\beta\lambda}
}\wedge dz^\lambda
+\paren{
	d\Phi_{\alpha\bar\mu}
	-\Phi_{\alpha\bar\beta}\ind{\theta}{\bar\mu}{\bar\beta}{}
	-\ind{\theta}{\alpha}{\beta}{}\Phi_{\beta\bar\mu}
}\wedge dz^{\bar\mu} \;.
\end{multline*}
The \emph{curvature form} $(\ind{\Theta}{\alpha}{\beta}{})$ of $\oKE$ is the collection of  $2$-forms 
\begin{equation*}
\ind{\Theta}{\alpha}{\beta}{}
	=d\ind{\theta}{\alpha}{\beta}{}-\ind{\theta}{\alpha}{\gamma}{}\wedge\ind{\theta}{\gamma}{\beta}{}
	=\ind{R}{\alpha}{\beta}{\lambda\bar\mu}\dz^\lambda\wedge\dz^{\bar\mu} 
\end{equation*}
where $\paren{\ind{R}{\alpha}{\beta}{\lambda\bar\mu}}$ stands for the curvature operator in the sense of $(\nabla_\lambda\nabla_{\bar\mu}-\nabla_{\bar\mu}\nabla_\lambda)\partial_\alpha=\ind{R}{\alpha}{\beta}{\lambda\bar\mu}\partial_\beta$. Since $d\Phi_{\alpha\bar\mu}-\Phi_{\alpha\bar\beta}\ind{\theta}{\bar\mu}{\bar\beta}{}-\ind{\theta}{\alpha}{\beta}{}\Phi_{\beta\bar\mu}=0$ by the metric compatibility \eqref{eqn:metric comp}, 
we get
\begin{align*}
-\ind{\Theta}{\alpha}{\beta}{}\Phi_\beta
=
\paren{
	d\Phi_{\alpha\lambda}
	-\Phi_{\alpha\beta}\ind{\theta}{\lambda}{\beta}{}
	-\ind{\theta}{\alpha}{\beta}{}\Phi_{\beta\lambda}
}\wedge\dz^\lambda
\;,
\end{align*}
equivalently,
\begin{equation*}
-\ind{R}{\alpha}{\beta}{\lambda\bar\mu}\Phi_\beta\dz^\lambda\wedge\dz^{\bar\mu}
=
	\Phi_{\alpha\lambda\mu}\dz^\mu\wedge\dz^\lambda
	+\Phi_{\alpha\lambda\bar\mu}\dz^{\bar\mu}\wedge\dz^\lambda
\end{equation*}
The curvature form $(\ind{\Theta}{\alpha}{\beta}{})$ consists of $(1,1)$-forms only so that
\begin{equation}\label{eqn:curvature}
\Phi_{\alpha\beta\gamma} =\Phi_{\alpha\gamma\beta} \;, \quad
\Phi_{\alpha\lambda\bar\mu} =\Phi_\beta\ind{R}{\alpha}{\beta}{\lambda\bar\mu} \;.
\end{equation}

\bigskip
	
Now differentiating $\norm{d\Phi}_{\oKE}^2=\Phi_\alpha\Phi^\alpha$, we have
\begin{multline*}
d\norm{d\Phi}_{\oKE}^2
	=(d\Phi_\alpha)\Phi^\alpha+\Phi_\alpha(d\Phi^\alpha)
	=
	\paren{d\Phi_\alpha-\ind{\theta}{\alpha}{\beta}{}\Phi_\beta}\Phi^\alpha
	+\Phi_\alpha\paren{d\Phi^\alpha+\Phi^\beta\ind{\theta}{\beta}{\alpha}{}}
	\\
	=
	\paren{
		\Phi_{\alpha\beta}\Phi^\alpha
		+\Phi_\alpha\ind{\Phi}{}{\alpha}{\beta}
	}\dz^\beta
	+\paren{
		\Phi_{\alpha\bar\beta}\Phi^\alpha
		+\Phi_\alpha\ind{\Phi}{}{\alpha}{\bar\beta}
	}\dz^{\bar\beta}
\end{multline*}
so that
\begin{equation}\label{eqn:first derivative}
\partial\norm{d\Phi}_{\oKE}^2 = \paren{
		\Phi_{\alpha\beta}\Phi^\alpha
		+\Phi_\alpha\ind{\Phi}{}{\alpha}{\beta}
	}\dz^\beta
	\;.
\end{equation}
We can get
\begin{multline*}
d\partial \norm{d\Phi}_{\oKE}^2
	=
	\paren{
		d\Phi_{\alpha\lambda}
		-\ind{\theta}{\alpha}{\gamma}{}\Phi_{\gamma\lambda}
		-\Phi_{\alpha\gamma}\ind{\theta}{\lambda}{\gamma}{}
	}\wedge\Phi^\alpha\dz^\lambda
	+\Phi_{\alpha\lambda}
	\paren{
		d\Phi^\alpha
		+\Phi^\gamma\ind{\theta}{\gamma}{\alpha}{}
	}\wedge\dz^\lambda
	\\
	+\paren{
		d\Phi_\alpha
		-\ind{\theta}{\alpha}{\gamma}{}\Phi_\gamma
	}\wedge\ind{\Phi}{}{\alpha}{\lambda}\dz^\lambda
	+\Phi_\alpha
	\paren{
		d\ind{\Phi}{}{\alpha}{\lambda}
		+\ind{\Phi}{}{\gamma}{\lambda}\ind{\theta}{\gamma}{\alpha}{}
		-\ind{\theta}{\lambda}{\gamma}{}\ind{\Phi}{}{\alpha}{\gamma}
	}\wedge\dz^\lambda
\end{multline*}
differentiating \eqref{eqn:first derivative}, then gathering $(1,1)$-forms we have
\begin{multline*}
\bar\partial\partial\norm{d\Phi}_{\oKE}^2
	=
	\Phi_{\alpha\lambda\bar\mu}\dz^{\bar\mu}\wedge\Phi^\alpha\dz^\lambda
	\\
	+\Phi_{\alpha\lambda}\ind{\Phi}{}{\alpha}{\bar\mu}\dz^{\bar\mu}\wedge\dz^\lambda
	+\Phi_{\alpha\bar\mu}\dz^{\bar\mu}\wedge\ind{\Phi}{}{\alpha}{\lambda}\dz^\lambda
	+\Phi_\alpha\ind{\Phi}{}{\alpha}{\lambda\bar\mu}\dz^{\bar\mu}\wedge\dz^\lambda
\end{multline*}
Since $\ind{\Phi}{}{\alpha}{\lambda\bar\mu}=h^{\alpha\bar\beta}\Phi_{\bar\beta\lambda\bar\mu}=0$,  
\begin{equation}\label{eqn:ddbar u}
\partial\bar\partial \norm{d\Phi}_{\oKE}^2 = -\bar\partial\partial \norm{d\Phi}_{\oKE}^2
	=
	\paren{
		\Phi_{\alpha\lambda\bar\mu}\Phi^\alpha
		+\Phi_{\alpha\lambda}\ind{\Phi}{}{\alpha}{\bar\mu}
		+\Phi_{\alpha\bar\mu}\ind{\Phi}{}{\alpha}{\lambda}
		}\dz^\lambda\wedge\dz^{\bar\mu}
	\;.
\end{equation}

Now we have
\begin{proposition}
If $\norm{d\Phi}_{\oKE}^2=\Phi_\alpha\Phi^\alpha$ is constant on $\Omega$, then
\begin{equation}\label{eqn:identity1}
\Phi_{\alpha\beta}\Phi^\alpha
		=-(n+1)\Phi_\beta
\end{equation}
and
\begin{equation}\label{eqn:identity2}
\Phi_{\alpha\lambda}\Phi^{\alpha\lambda}
		=(n+1)\Phi_\alpha\Phi^\alpha
		-n(n+1)^2
		\;.
\end{equation}
\end{proposition}
\begin{proof}
Since $\norm{d\Phi}_{\oKE}^2$ is constant, so $\partial \norm{d\Phi}_{\oKE}^2=0$ and $\partial\bar\partial \norm{d\Phi}_{\oKE}^2=0$. From \eqref{eqn:first derivative} and \eqref{eqn:ddbar u}, we have
\begin{equation*}
\Phi_{\alpha\beta}\Phi^\alpha = -\Phi_\alpha\ind{\Phi}{}{\alpha}{\beta} = -\Phi^{\bar\gamma}\Phi_{\bar\gamma\beta} = -(n+1)\Phi^{\bar\gamma}h_{\bar\gamma\beta} = -(n+1)\Phi_\beta \;,
\end{equation*}
and
\begin{equation*}
\Phi_{\alpha\lambda\bar\mu}\Phi^\alpha
		+\Phi_{\alpha\lambda}\ind{\Phi}{}{\alpha}{\bar\mu}
		+\Phi_{\alpha\bar\mu}\ind{\Phi}{}{\alpha}{\lambda}
		=0 \;.
\end{equation*}
From the second identity in \eqref{eqn:curvature} and $\Phi_{\alpha\bar\beta}=(n+1)h_{\alpha\bar\beta}$, the last identity can be written by
\begin{equation*}
\Phi_\beta\ind{R}{\alpha}{\beta}{\lambda\bar\mu}\Phi^\alpha
		+\Phi_{\alpha\lambda}\ind{\Phi}{}{\alpha}{\bar\mu}
		+(n+1)^2h_{\lambda\bar\mu}
		=0 \;.
\end{equation*}
Contracting by $h^{\lambda\bar\mu}$, we have
\begin{equation*}
-(n+1)\Phi_\alpha\Phi^\alpha
		+\Phi_{\alpha\lambda}\Phi^{\alpha\lambda}
		+n(n+1)^2
		=0 \;.
\end{equation*}
from the Einstein condition $R_{\alpha\bar\beta\lambda\bar\mu}h^{\lambda\bar\mu}=R_{\alpha\bar\beta}=-(n+1)h_{\alpha\bar\beta}$. This completes the proof.
\end{proof}

%%%%%%%%%%%%%%%%%%%%%%%%%%%%%%%%%%%%%%%%
%%%%%%%%%%%%%%%%%%%%%%%%%%%%%%%%%%%%%%%%
%%%%%%%%%%%%%%%%%%%%%%%%%%%%%%%%%%%%%%%%
%%%%%%%%%%%%%%%%%%%%%%%%%%%%%%%%%%%%%%%%
%%%%%%%%%%%%%%%%%%%%%%%%%%%%%%%%%%%%%%%%
%%%%%%%%%%%%%%%%%%%%%%%%%%%%%%%%%%%%%%%%

\subsection{Proof of Theorem~\ref{thm:main thm2}}
Now suppose that $\Phi=\Omega\to\RR$ satisfies 
\begin{equation*}
d\dc\Phi=(n+1)\oKE \quad\text{and}\quad
\norm{d\Phi}_{\oKE}\equiv C
\end{equation*}
for some $C>0$. We first consider the vector field $\vV$ of type $(1,0)$ defined by
\begin{equation}\label{eqn:generator of L}
\vV=\Phi^\alpha\partial_\alpha = h^{\alpha\bar\beta}\Phi_{\bar\beta}\pd{}{z^\alpha} \;.
\end{equation}
This $\vV$ has positive constant length: $\norm{\vV}_{\oKE}^2=\Phi^\alpha\Phi^{\bar\beta}h_{\alpha\bar\beta}\equiv C^2$. Thus let us consider the line bundle $L\to\Omega$ generated by $\vV$:
\begin{equation*}
L=\set{v\in T^{1,0}\Omega: \text{$v$ is parellel to $\vV$}} \;.
\end{equation*}
This is a subbudle of $T^{1,0}\Omega$ but not holomorphic in general.

\begin{proposition}
If $\norm{d\Phi}_{\oKE}\equiv C$ for some constant $C$ with $0<C\leq(n+1)$, then $L$ is holomorphic.
\end{proposition}
\begin{proof}
We will prove that there is a real number $t$ such that the nowhere vanishing section $e^{t\Phi}\vV$ to $L$ is a holomorphic tangent vector field. Let us consider
\begin{equation*}
\partial_{\bar\beta}\paren{e^{t\Phi}\Phi^\alpha} 
	= te^{t\Phi}\paren{\partial_{\bar\beta}\Phi} \Phi^\alpha
	+e^{t\Phi}\partial_{\bar\beta}\Phi^\alpha
	=e^{t\Phi}\paren{
		t\paren{\partial_{\bar\beta}\Phi} \Phi^\alpha+\partial_{\bar\beta}\Phi^\alpha
		}
	\;.
\end{equation*}
Since $\partial_{\bar\beta}\Phi = \Phi_{\bar\beta}$ and $\ind{\Phi}{}{\alpha}{\bar\beta} = \partial_{\bar\beta}\Phi^\alpha+\ind{\theta}{\gamma}{\alpha}{}(\partial_{\bar\beta})\Phi^\gamma= \partial_{\bar\beta}\Phi^\alpha$, we have
\begin{equation*}
\nabla_{\bar\beta}(e^{t\Phi}\vV)
	=\nabla_{\bar\beta}\paren{e^{t\Phi}\Phi^\alpha\partial_\alpha}
	=\partial_{\bar\beta}\paren{e^{t\Phi}\Phi^\alpha}\partial_\alpha
	=e^{t\Phi}\paren{
		t\Phi_{\bar\beta} \Phi^\alpha+\ind{\Phi}{}{\alpha}{\bar\beta}
		}\partial_\alpha
		\;.
\end{equation*}
When we denote $\nabla''$ by the $(0,1)$-part of $\nabla$, it follows
\begin{equation*}
\nabla''\paren{e^{t\Phi}\vV}
	=\nabla_{\bar\beta}\paren{e^{t\Phi}\vV}\otimes\dz^{\bar\beta}
	=e^{t\Phi}\paren{
		t\Phi_{\bar\beta} \Phi^\alpha+\ind{\Phi}{}{\alpha}{\bar\beta}
		}\partial_\alpha\otimes\dz^{\bar\beta}
		\;.
\end{equation*}
The holomorphicity of $e^{t\Phi}\vV$ is equivalent to the vanishing length of $\nabla''\paren{e^{t\Phi}\vV}$ with respect to $\oKE$ that can be computed by
\begin{multline*}
\norm{\nabla''\paren{e^{t\Phi}\vV}}_{\oKE}^2
	= e^{2t\Phi}\paren{
		t\Phi_{\bar\beta} \Phi^\alpha+\ind{\Phi}{}{\alpha}{\bar\beta}
		}
		\paren{
		t\Phi^{\bar\beta} \Phi_\alpha+\ind{\Phi}{\alpha}{\bar\beta}{}
		}
	\\
	= e^{2t\Phi}\paren{
		t^2\Phi_{\bar\beta} \Phi^\alpha\Phi^{\bar\beta} \Phi_\alpha
		+t\Phi_{\bar\beta} \Phi^\alpha\ind{\Phi}{\alpha}{\bar\beta}{}
		+t\ind{\Phi}{}{\alpha}{\bar\beta}\Phi^{\bar\beta} \Phi_\alpha
		+\ind{\Phi}{}{\alpha}{\bar\beta}\ind{\Phi}{\alpha}{\bar\beta}{}
		}
	\\
	= e^{2t\Phi}\paren{
		t^2 \Phi_\alpha\Phi^\alpha\Phi_{\bar\beta} \Phi^{\bar\beta}
		+t\Phi^{\beta} \Phi^\alpha\Phi_{\alpha\beta}
		+t\Phi_{\bar\alpha\bar\beta}\Phi^{\bar\beta} \Phi^{\bar\alpha}
		+\Phi^{\alpha\beta}\Phi_{\alpha\beta}
		}
	\;.
\end{multline*}
Note that $\Phi_\alpha\Phi^\alpha=\Phi_{\bar\beta} \Phi^{\bar\beta}\equiv C^2$ by the assumption and 
\begin{equation*}
\Phi^{\beta} \Phi^\alpha\Phi_{\alpha\beta}=-(n+1)\Phi^{\beta} \Phi_\beta\equiv-(n+1)C^2
\end{equation*}
by \eqref{eqn:identity1}. From \eqref{eqn:identity2}, we have 
\begin{equation*}
\Phi^{\alpha\beta}\Phi_{\alpha\beta}
	\equiv(n+1)C^2-n(n+1)^2
	\;.
\end{equation*} 
Therefore
\begin{equation*}
\norm{\nabla''\paren{e^{t\Phi}\vV}}_{\oKE}^2
	\equiv e^{2t\Phi}\paren{
		t^2 C^4
		-t2(n+1)C^2
		+(n+1)C^2-n(n+1)^2
		}
		\;.
\end{equation*}
So $e^{t\Phi}\vV$ is holomorphic if and only if $t$ satisfies
\begin{equation*}
 C^4t^2-2(n+1)C^2t+(n+1)C^2-n(n+1)^2=0\;.
\end{equation*}
The discriminant of this quadratic polynomial  is
\begin{equation*}
(n+1)^2C^4-C^4((n+1)C^2-n(n+1)^2)
	=
	(n+1)C^4((n+1)^2-C^2)
	\;.
\end{equation*}
This completes the proof.
\end{proof}

The complete holomorphic vector field in Theorem~\ref{thm:main thm2} is indeed a holomorphic  section to $L\to\Omega$. Let us consider the function $-e^{-\ve\Phi}$. Since
\begin{multline*}
d\dc\paren{-e^{-\ve\Phi}}=-i\partial\bar\partial e^{-\ve\Phi}
	= -i\partial\paren{
		-\ve e^{-\ve\Phi} \bar\partial\Phi
		}
	\\
	= \ve e^{-\ve\Phi}\paren{
	i\partial\bar\partial\Phi
	-\ve  i\partial\Phi\wedge\bar\partial\Phi
	}
	= (n+1)\ve e^{-\ve\Phi}\paren{
	\oKE
	 -\frac{\ve}{(n+1)} i\partial\Phi \wedge  \bar\partial\Phi
	}\;,
\end{multline*}
the Cauchy-Schwarz inequality implies that $-e^{-\ve\Phi}$ is strictly plurisubharmonic for any $\ve<(n+1)/C^2=(n+1)/\norm{d\Phi}_{\oKE}^2$. If $\ve=(n+1)/C^2$, then the vector field $\vV$ annihilates $d\dc(-e^{-\ve\Phi})$ in the sense of
\begin{equation*}
\vV\inp d\dc\paren{-e^{-(n+1)\Phi/C^2}}\equiv 0
\end{equation*}
because $\vV\inp\oKE=ih_{\alpha\bar\beta}\Phi^\alpha\dz^{\bar\beta}=i\Phi_{\bar\beta}\dz^{\bar\beta}=i\bar\partial\Phi$ and $\partial\Phi(\vV)=\Phi_\alpha\Phi^\alpha=C^2$ so that
\begin{multline*}
\vV\inp d\dc\paren{-e^{-(n+1)\Phi/C^2}}
	= \frac{(n+1)^2}{C^2} e^{-(n+1)\Phi/C^2}\paren{
	\vV\inp\oKE
	 -\frac{1}{C^2} i\partial\Phi(\vV)\bar\partial\Phi
	}
	\\
	= \frac{(n+1)^2}{C^2} e^{-(n+1)\Phi/C^2}\paren{
	i\bar\partial\Phi
	 -i\bar\partial\Phi
	}
	=0
	\;.
\end{multline*}
\begin{proposition}
Let 
\begin{equation*}
\rho=-e^{-(n+1)\Phi/C^2} \;.
\end{equation*}
For any (local) holomorphic section $\vZ$ to $L$, the function $\vZ\rho$ is holomorphic.
\end{proposition}
\begin{proof}
Let $\vW$ be a holomorphic tangent vector field of $\Omega$. We will prove that $\overline{\vW}(\vZ\rho)\equiv 0$. For any properly differentiable function $g:\Omega\to\RR$, $\vZ(\overline \vW g)=\overline \vW (\vZ g)$. Therefore $\vZ\overline \vW=\overline \vW \vZ$ as operators, so $[\vZ,\overline \vW]=0$. Since
\begin{multline*}
dd^c\rho (\vZ,\overline \vW)
	=\vZ \big(\dc\rho (\overline \vW)\big)- \overline \vW \big(\dc\rho(\vZ)\big)-\dc\rho([\vZ,\overline \vW])
	\\
	=\vZ \big(\dc\rho (\overline \vW)\big)- \overline \vW \big(\dc\rho(\vZ)\big)
\;,
\end{multline*}
and the section $\vZ$ to $L$ annihilates $dd^c\rho$, we have
\begin{equation*}
\overline \vW \big(\dc\rho(\vZ)\big) = \vZ \big(\dc\rho (\overline \vW)\big)\;.
\end{equation*}
Since $\dc=\displaystyle{\frac{i}{2}(\bar\partial-\partial)}$ and $\vZ\overline \vW=\vW\overline \vZ$, it follows that
\begin{equation*}
 -\frac{i}{2}\overline \vW (\vZ\rho)= \frac{i}{2}\vZ (\overline \vW \rho) =  \frac{i}{2}\overline \vW ( \vZ\rho)\;.
\end{equation*}
This implies that $\overline \vW ( \vZ\rho)\equiv0$, so as a conclusion $\vZ\rho$ is holomorphic.
\end{proof}

Note that the function $\vV\rho$ is nowhere vanishing since
\begin{multline}\label{eqn:length of Vrho}
\vV\rho= \Phi^\alpha\partial_\alpha \paren{-e^{-(n+1)\Phi/C^2}}
	=\Phi^\alpha\paren{e^{-(n+1)\Phi/C^2} \frac{n+1}{C^2}\Phi_\alpha}
	\\
	=(n+1)e^{-(n+1)\Phi/C^2}=-(n+1)\rho>0
	\;.
\end{multline}
If $\vW$ is a nowhere vanishing local holomorphic section to $L$, then there is non-vanishing smooth function $g$ such that $\vW=g\vV$. Therefore $\vW\rho=g(\vV\rho)$ is a holomorphic function which is nowhere vanishing on its domain. Then we can define the holomorphic vector field $\widetilde{\vW}$ by
\begin{equation*}
\widetilde{\vW}=\frac{i}{\vW\rho} \vW \;.
\end{equation*}
If $\vW'$ is another nonvanishing holomorphic section to $L$, then $\vW'=g \vW$ for some nonvanishing holomorphic function $g$ on an open set where $\vW$ and $\vW'$ are both defined. Moreover
\begin{equation*}
\widetilde \vW'= \frac{i}{\vW'\rho}\vW'=\frac{i}{g\vW\rho}g\vW = \frac{i}{\vW\rho} \vW =\widetilde{\vW} \;.
\end{equation*}
Therefore we can define a global holomorphic vector field $\vZ_\rho$ of $\Omega$ by
\begin{equation}\label{def:vf}
\vZ_\rho=\frac{i}{\vW\rho} \vW
\end{equation}
for any nowhere vanishing holomorphic section $\vW$ to $L$.

\medskip

The following implies Theorem~\ref{thm:main thm2}.

\begin{proposition}
The vector field $\vZ_\rho$ in \eqref{def:vf} is complete.
\end{proposition}
\begin{proof}
The real part of $\vZ_\rho$ is tangent to $\rho$:
\begin{equation*}
(\RE \vZ_\rho)\rho = (\vZ_\rho+\overline{\vZ_\rho})\rho 
	= \paren{\frac{i}{\vW\rho} \vW-\frac{i}{\overline{\vW}\rho} \overline{\vW}}\rho 
	= 0\;.
\end{equation*}
The length of $\vZ_\rho$ can be locally written by
\begin{equation*}
\norm{\vZ_\rho}_{\oKE}^2 
	= \norm{\frac{i}{\vW\rho} \vW}_{\oKE}^2 
	= \frac{1}{\abs{\vW\rho}^2}\norm{\vW}_{\oKE}^2\;.
\end{equation*}
When we let $\vW=g\vV$ for some $g$,  we have $\norm{\vW}_{\oKE}^2=g^2 \norm{\vV}_{\oKE}^2=g^2 C^2$ and $\abs{\vW\rho}^2= (n+1)^2g^2\rho^2$ from \eqref{eqn:length of Vrho}. This means that
\begin{equation*}
\norm{\vZ_\rho}_{\oKE}^2 = \frac{g^2 C^2}{(n+1)^2g^2\rho^2}=\frac{C^2}{(n+1)^2\rho^2} \;.
\end{equation*}
This implies that $\rho \vZ_\rho$ has constant length $C/(n+1)$. Since the \KE metric $\oKE$ is complete, the vector field $\RE( \rho \vZ_\rho)=\rho(\RE \vZ_\rho)$ is complete.

In order to show the completeness of $\vZ_\rho$, take any integral curve $\gamma:\RR\to X$ of $\rho (\RE\vZ_\rho)$. It satisfies 
\begin{equation*}
\paren{\rho(\RE\vZ_\rho)}\circ \gamma = \dot\gamma
\end{equation*}
equivalently
\begin{equation*}
(\RE \vZ_\rho)\circ \gamma = \paren{\rho^{-1}\circ\gamma} \dot\gamma 
\end{equation*}
Since $(\RE \vZ_\rho)\rho\equiv 0$, equivalently $\paren{\rho(\RE \vZ_\rho)}\rho\equiv 0$,  the curve $\gamma$ lies on a level set of $\rho$ so $\rho^{-1}\circ\gamma\equiv c$ for some  negative constant $c$. For the curve $\sigma:\RR\to X$ defined by $\sigma(t)=\gamma(ct)$, we have 
\begin{equation*}
(\RE \vZ_\rho)\circ \sigma (t)  = (\RE \vZ_\rho)(\gamma(ct)) = c\dot\gamma(ct) =\dot\sigma(t)
\end{equation*}
This means that $\sigma:\RR\to\Omega$ is the integral curve of $\RE \vZ_\rho$; therefore $\RE \vZ_\rho$ is complete. This completes the proof.
\end{proof}

%\bibliographystyle{siam} %alpha, plain, abbrv, siam
%\bibliography{convergenceVIK}

\end{document}